\documentclass[12pt,twoside]{article}
\usepackage{amssymb,amsthm}
\usepackage[namelimits,sumlimits,fleqn]{amsmath}
\usepackage{setspace}
\usepackage[ hmargin={2.2cm,2.4cm}, top=36mm,
    a4paper,pdftex, textheight=230mm,
    headheight=20mm, headsep=10mm, bindingoffset=0.9cm]
    {geometry}
\usepackage[small, margin=20pt]{caption}
\usepackage[usenames,dvipsnames]{color}
\definecolor{darkblue}{rgb}{0.0,0.0,0.6}
\usepackage{float}\restylefloat{figure}
\usepackage{url}
\usepackage[backref=page]{hyperref}
\hypersetup{colorlinks=true,breaklinks=true,
            linkcolor=darkblue,urlcolor=darkblue,
            anchorcolor=darkblue,citecolor=darkblue}

\usepackage{fancyhdr}
\allowdisplaybreaks[4]

\newcommand{\ip}[2]{\ensuremath{\langle #1,#2\rangle}}

\title{A Point-Line Incidence Identity in Finite Fields, and Applications}
\author{Brendan Murphy and Giorgis Petridis}
\date{}

\theoremstyle{plain}
\newtheorem{theorem}{Theorem}
\newtheorem{lemma}{Lemma}

\newtheorem{corollary}{Corollary}

\theoremstyle{definition}
\newtheorem*{acknowledgement}{Acknowledgement}
\newtheorem*{unremark}{Remark}
\newtheorem*{claim}{Claim}

\renewcommand*{\backref}[1]{}
\renewcommand*{\backrefalt}[4]{%
    \ifcase #1 (Not cited.)%
    \or        (Cited on page~#2.)%
    \else      (Cited on pages~#2.)%
    \fi}

\begin{document}

\onehalfspacing

\pagenumbering{arabic}

\setcounter{section}{0}

\bibliographystyle{plain}

\maketitle

\begin{abstract}
Let $E \subseteq \mathbb{F}_q^2$ be a set in the 2-dimensional vector space over a finite field with $q$ elements. We prove an identity for the second moment of its incidence function and deduce a variety of existing results from the literature, not all naturally associated with lines in $\mathbb{F}_q^2$, in a unified and elementary way.
\end{abstract}

\section[Introduction]{Introduction}
\label{Introduction}

\let\thefootnote\relax\footnotetext{The second author is supported by the USA NSF DMS Grant 1500984.}

The first lemma in the breakthrough paper of Bourgain, Katz, and Tao \cite{BKT2004} on the sum-product phenomenon in finite fields states that if $A, B \subseteq \mathbb{F}_q$ are sets in a finite field with $q$ elements, then there exists $0 \neq \xi \in \mathbb{F}_q^*$ such that the set
\[
A + \xi B = \{a + \xi b : a \in A, b \in B\}
\]
has cardinality at least a constant multiple of the minimum of $q$ and $|A| |B|$. 

Interpreted geometrically, the lemma states that for all Cartesian products $E = A \times B \subseteq \mathbb{F}_q^2$ in the 2-dimensional vector space over $\mathbb{F}_q$, there exists a direction $(1, \xi)$ such that the projection of $E$ onto this direction, denoted by $E \cdot (1,\xi)$, is about as large it can be
\[
|E \cdot (1,\xi)| = |\{ v \cdot (1,\xi) : v \in E\}| = \Omega(\min\{q, |E|\}).
\]

The purpose of this note is to use the simple observation that both the statement and the proof of the lemma can be generalised to general sets $E \subseteq \mathbb{F}_q^2$ and deduce in a unified way various sum-product results in the literature, even some that on the surface do not appear to be associated with lines in $\mathbb{F}_q^2$. 

A closer inspection of the second proof provided by Bourgain, Katz, and Tao reveals that the authors are computing the second moment of the \emph{incidence function} associated with  $ A \times B \subseteq \mathbb{F}_q^2$. We consider the incidence function for all sets $E \subseteq \mathbb{F}_q^2$. Given a line $\ell \subset \mathbb{F}_q^2$ we denote by $i(\ell)$ the number of incidences of $\ell$ with $E$
\begin{equation}\label{i}
i_E(\ell) = i(\ell) = | \ell \cap E|.
\end{equation}

The key observation is the following identity on the second moment of the incidence function $i$. 
\begin{lemma} \label{var i}
Let $E\subseteq F_q^2$ and let $i$ be the incidence function defined in \eqref{i}.
The following identity holds.
\[
\sum_\ell i(\ell)^2 = |E|^2 +  q|E|.
\]
Hence
\[
\sum_\ell \left(i(\ell) - \frac{|E|}{q}\right)^2  \leq  q |E|.
\] 
Both sums are over all lines in $\mathbb{F}_q^2$. 
\end{lemma}

We provide a proof in Section~\ref{Main lemma} and prove a generalisation to higher dimensions in Section~\ref{HDim}. In the remaining sections we derive variants of existing results from the literature. A detailed review of the relevant bibliography is given in each section.
\begin{enumerate}
\item \emph{Discrete Marstrand}. A generalisation of the Bourgain-Katz-Tao lemma to all sets, which can also be thought of as a discrete version of a classical theorem of Marstrand on projections in Euclidean space \cite{Marstrand1954}: Let $E \subseteq \mathbb{F}_q^2$. There exists $0 \neq \xi \in \mathbb{F}_q^*$ such that
\[
|E \cdot (1, \xi) | \geq \frac{1}{2} \min\{|E|, q\}.
\] 
This result was communicated to the authors by Alex Iosevich and is proved in Section~\ref{Maarstrand}.
\item \emph{Vinh's point-line incidence bound}. A point-line incidence theorem of Vinh \cite{Vinh2011}: Let $E \subseteq \mathbb{F}_q^2$ and $\mathcal{L}$ be a collection of lines in $\mathbb{F}_q^2$. The number
\[
I(E,\mathcal{L}) = \sum_{\ell \in \mathcal{L}} i(\ell)
\] 
of point-line incidences satisfies
\[
\left| I(E, \mathcal{L}) - \frac{|E||\mathcal{L}|}{q} \right| \leq \sqrt{q |E| |\mathcal{L}|}.
\] 
Proved in Section~\ref{Vinh}.
\item \emph{Pinned dot products}. A slightly stronger version of a theorem of Chapman, Erdo{\u{g}}an, Hart, Iosevich, and Koh on pinned dot products from \cite{CEHIK2012}. Let $E \subseteq \mathbb{F}_q^2$ and $\mathrm{Dir}(E) \subseteq \mathbb{F}_q$ be the set of directions determined by $E$ (that is the set of lines through the origin incident to $E$, c.f. Section~\ref{Pinned dot prod} just above Theorem~\ref{Pinned dot} on p.\pageref{Pinned dot}). Suppose that $|E| |\mathrm{Dir}(E)|> q^{2}$. There exists $e \in E$ such that
\[
E \cdot e = \{ u \cdot e : u \in E\}
\]
contains half the elements of $\mathbb{F}_q$.

We also investigate the special case where $E = A \times A$ and prove statements like the following result on pinned dot products of Cartesian products in $\mathbb{F}_q^2$. Let $A \subseteq \mathbb{F}_q$ and suppose that $|A|^2 |AA^{-1}| > q^{2}$, where $AA^{-1} = \{ab^{-1}: a,b \in A\}$. There exists $a,b \in A$ such that
\[
\frac{q}{2} \leq |aA+bA| = |\{ ac+bd : c,d \in A\}|.
\]
Proved in Section~\ref{Pinned dot prod}.
\item \emph{Pinned algebraic distances}. The algebraic distance between two points $u=(u_1,u_2) ,v=(v_1,v_2) \in \mathbb{F}_q^2$ is defined as
\[
\| u - v\| = (u-v) \cdot (u-v) = (u_1-v_1)^2 + (u_2-v_2)^2.
\] 

We prove some results on the set of algebraic distances determined by $E\subseteq \mathbb{F}_q^2$ similar to those in \cite{Iosevich-Rudnev2007,  CEHIK2012, HLRN2016}. Let $E \subseteq \mathbb{F}_q^2$ and $\ell$ be a line in $\mathbb{F}_q^2$. Suppose that $|E| |E \cap \ell | > 2 q^{2}$. There exists $e \in E\cap \ell$ such that the pinned distance set
\[
\{ \| v - e \| : v \in E\}
\]
contains half the elements of $\mathbb{F}_q$.

In the case where $E = A \times A$ for $A \subseteq \mathbb{F}_q$ and $\ell$ is the span of $(1,1)$, the above becomes: $|A|^3 > 2 q^{2}$ implies there exists $a \in A $ such that
\[
\frac{q}{2} \leq |(a-A)^2+(a-A)^2| = |\{ (a-b)^2+(a-c)^2 : b,c \in A\}|.
\]
We also prove a slightly stronger variant. Suppose that $|A|^2 |D| > 2q^{2}$, where $D=A-A = \{a - b: a,b \in A\}$. There exists $a \in A$ such that
\[
\frac{q}{2} \leq |(a-A)^2+D^2| = |\{ (a-b)^2+s^2 : b \in A, s \in D\}|.
\]
Proved in Section~\ref{Iosevich-Rudnev}.
\item \emph{Higher dimensions}. We also prove some similar in nature results of Chapman, Erdo{\u{g}}an, Hart, Iosevich, and Koh from \cite{CEHIK2012} on $\mathbb{F}_q^d$ (the $d$-dimensional vector space over $\mathbb{F}_q$) for $d\geq 2$ in Section~\ref{HD}.
\item \emph{Generalisation to block designs}. We generalise Lemma~\ref{var i} to the setting of \emph{block designs}. In \cite{lund2014incidence}, Lund and Saraf generalised Vinh's method to block diagrams. We show that the same elementary method that reproves Vinh's point-line incidence bound can be used to reprove Lund and Saraf's incidence bound for block designs.

\item \emph{Comparison of our results to related methods}.
In the second to last section, we compare our method to the spectral graph theory method of Vinh \cite{Vinh2011} and Lund and Saraf \cite{lund2014incidence}.
We also show how the graph theoretic argument can be simplified to avoid the use of eigenvalues or singular values; in fact, we will see that a key inequality used to prove the ``expander mixing lemma'' is actually an \emph{equality}\/ in this context.

In the final section, we show how our elementary argument for block designs can be modified to yield a bound of Cilleruelo on the cardinality of Sidon sets \cite{Cilleruelo2012}.
\end{enumerate}
\begin{acknowledgement}
The authors would like to than Alex Iosevich, Jonathan Pakianathan, and Misha Rudnev for helpful conversations. They would also like to thank the referee for an insightful report.
\end{acknowledgement}

\section[Proof of Lemma~1 and a corollary]{Proof of Lemma~\ref{var i} and a corollary}
\label{Main lemma}

We first prove Lemma~\ref{var i} and then draw a useful quantitative consequence. The proof is a standard second moment calculation driven by the fact that any collection of lines in $\mathbb{F}_q^2$ is a pseudorandom collection. That is, for any two distinct lines $\ell, \ell'$ we have $\displaystyle \frac{|\ell \cap \ell'|}{q^2} \leq \frac{1}{q^2} = \frac{1}{q} \, \frac{1}{q} = \frac{|\ell |}{q^2} \, \frac{|\ell'|}{q^2}$.

\begin{proof}[Proof of Lemma~\ref{var i}]
Sums are over all lines in $\mathbb{F}_q^2$.  We denote by $\ell$ the characteristic function of a line $\ell$.
\begin{align*}
\sum_\ell i(\ell)^2 
& = \sum_\ell \left( \sum_{v \in E} \ell(v) \right)^2 \\
& = \sum_\ell  \sum_{v,v' \in E} \ell(v) \ell(v') \\
& = \sum_{v \in E} \sum_\ell \ell(v) + \sum_{v\neq v' \in E} \sum_{\ell} \ell(v) \ell(v')\\
& =  |E| (q+1) + |E| (|E| - 1) \\
& = |E|^2 + q |E|.
\end{align*}
In the penultimate line we used the facts that $q+1$ lines are incident to a point and that two distinct points are incident to a unique line.

The derivation of the second conclusion is similar to the proof of the well-known expression for variance. 
\begin{align*}
\sum_\ell \left(i(\ell) - \frac{|E|}{q}\right)^2 
& = \sum_\ell i(\ell)^2 - 2 \frac{|E|}{q} \sum_\ell i(\ell) + q(q+1) \frac{|E|^2}{q^2} \\
& = \sum_\ell i(\ell)^2 - 2 \frac{|E|}{q} (q+1) |E| +(q+1) \frac{|E|^2}{q}\\
& = \sum_\ell i(\ell)^2 - |E|^2 - \frac{|E|^2}{q} \\
& \leq \sum_\ell i(\ell)^2 - |E|^2 \\
& =  q |E|.
\end{align*}
\end{proof}

Next, we deduce a quantitative version of the statement that if $E \subseteq \mathbb{F}_q^2$ and $\Theta \subseteq \mathbb{F}_q$ are ``large'', then there is a direction  $(1,\theta)$ for some $\theta \in \Theta$ such that the projection of $E$ onto $(1,\theta)$ has (nearly maximum) cardinality $\Omega(q)$. 

\begin{corollary} \label{Good direction}
Let $E \subseteq \mathbb{F}_q^2$ and $\Theta \subseteq \mathbb{F}_q$. There exists $\theta \in \Theta$ such that
\[
|E \cdot (1, \theta)|  \geq q \, \frac{|E| |\Theta|}{q^2 + |E| |\Theta|}.
\]
Hence $|E| |\Theta| > q^2$ implies there exists $\theta \in \Theta$ such that $|E \cdot (1, \theta)| > q/2$.
\end{corollary}

\begin{proof}
We show that Lemma~\ref{var i} implies that $E$ is equidistributed on lines orthogonal to some direction in $\Theta$. The Cauchy-Schwartz inequality then implies the conclusion. 

Let us write $v_\theta = (1, \theta)$ for each $\theta \in \Theta$ and $\ell_{\theta,t} = \{ x \in \mathbb{F}_q^2 : x \cdot v_\theta = t\}$ for the line in $\mathbb{F}_q^2$ orthogonal to $v_\theta$ and incident to $(t,0)$. Note that the lines $\ell_{\theta,t}$ are distinct. The second statement in Lemma~\ref{var i} implies
\[
\sum_{\theta \in \Theta} \sum_{t \in \mathbb{F}_q} \left(i(\ell_{\theta,t}) - \frac{|E|}{q}\right)^2 \leq \sum_{\ell} \left(i(\ell) - \frac{|E|}{q}\right)^2 \leq q|E|.
\] 
Therefore there exists $\theta \in \Theta$ such that 
\[
\sum_{t \in \mathbb{F}_q} \left(i(\ell_{\theta,t}) - \frac{|E|}{q}\right)^2 \leq \frac{q|E|}{|\Theta|}.
\]
Expanding the square on the left side gives
\[
\sum_{t \in\mathbb{F}_q} i(\ell_{\theta,t})^2 - 2 \frac{|E|}{q} \sum_{t \in \mathbb{F}_q} i(\ell_{\theta,t}) + \frac{|E|^2}{q} = \sum_{t \in\mathbb{F}_q} i(\ell_{\theta,t})^2 - 2 \frac{|E|^2}{q} + \frac{|E|^2}{q} = \sum_{t \in E \cdot v_\theta} i(\ell_{\theta,t})^2 - \frac{|E|^2}{q}.
\]
Above, we used the fact that the collection of lines $\{ \ell_{\theta,t}\}_{t \in \mathbb{F}_q}$ partitions $\mathbb{F}_q^2$ and that $i(\ell_{\theta,t}) = 0$ when $t \notin E \cdot v_\theta$. Substituting above gives
\[
\sum_{t \in E \cdot v_\theta} i(\ell_{\theta,t})^2 \leq \frac{|E|^2}{q} + \frac{q|E|}{|\Theta|}.
\]
By the Cauchy-Schwarz inequality the left side is bounded below by
\[
\frac{\left( \sum_{t \in E \cdot v_\theta} i(\ell_{\theta,t}) \right)^2}{|E \cdot v_\theta|} = \frac{|E|^2}{|E \cdot v_\theta|}.
\]
Some algebraic manipulations yield $\displaystyle |E \cdot v_\theta| \geq \frac{q |E| |\Theta|}{q^2 + |E| |\Theta|}.$ 

The second conclusion follows from the fact that the above lower bound for $|E \cdot v_\theta|$ is an increasing function on the quantity $|E| |\Theta|$ and so is minimised at the minimum value of $|E| |\Theta|$.
\end{proof}

\section{Higher dimensions}
\label{HDim}

The arguments of the previous section work equally well in higher dimensions. As we will briefly mention high-dimensional applications of our method, we give full proofs of analogous statements to those of the previous section.

In $\mathbb{F}_q^d$ we deal with a set of points $E \subseteq \mathbb{F}_q^d$ and a collection of hyperplanes $\mathcal{H}$. A hyperplane is simply a translate of a $(d-1)$-dimensional subspace and is defined algebraically by $\{ v : v \cdot e = t\}$ for some $e\in \mathbb{F}_q^d$ and $t \in \mathbb{F}_q$. Note that there are 
\[
q^{d-1}+q^{d-2} + \dots + 1 = \frac{q^d-1}{q-1}
\] 
choices for $e$, each yielding $q$ distinct hyperplanes and hence making the total number of hyperplanes equal to $\frac{q(q^{d}-1)}{q-1}$. A natural way to count the possible choices of $e$ is to write $*$ for an arbitrary element for $\mathbb{F}_q$ and note that $e$ takes one of the following forms $(*,\dots,*,*,1), (*,\dots,*,1,0),\dots, (*,1,0,\dots,0), (1,0,\dots,0)$. Before moving on note that for each such $e$, the collection of hyperplanes $\{ x : e \cdot x = t\}_{t\in \mathbb{F}_q}$ partitions $\mathbb{F}_q^d$. We call such hyperplanes orthogonal to $e$.
 
Let us now prove a higher dimensional analogue of Lemma~\ref{var i}. The incidence function $i$ counts the incidences between a fixed point set $E$ and hyperplanes $h$. 
\begin{lemma} \label{var i HD}
Let $d\geq 2$, $E\subseteq F_q^d$, and $i$ be the incidence function defined by $i(h) = |E \cap h|$, where $h \subset \mathbb{F}_q^d$ is a hyperplane. The following identity holds.
\[
\sum_h i(h)^2 = \frac{q^{d-1}-1}{q-1} |E|^2 + q^{d-1} |E|.
\]
Hence
\[
\sum_h \left(i(h) - \frac{|E|}{q}\right)^2  \leq  q^{d-1} |E|.
\] 
The sums are over all hyperplanes in $\mathbb{F}_q^d$. 
\end{lemma}
\begin{proof}
We denote by $h$ the characteristic function of a hyperplane $h$.
\begin{align*}
\sum_h i(h)^2 
& = \sum_h \left( \sum_{v \in E} h(v) \right)^2 \\
& = \sum_h  \sum_{v,v' \in E} h(v) h(v') \\
& = \sum_{v \in E} \sum_h h(v) + \sum_{v\neq v' \in E} \sum_{h} h(v) h(v').
\end{align*}
There are $\frac{q^d-1}{q-1}$ hyperplanes incident to each $v \in E$ (precisely one for each ``direction $e$''). So the first summand equals $\frac{q^d-1}{q-1} |E|$.

For the second summand we count how many hyperplanes are incident to both $v$ and $v'$ for $v \neq v' \in E$. Each such hyperplane in characterised by a direction $e$ such that $e \cdot (v - v') = 0$. This is because if both $v$ and $v'$ belong to $\{ x : x \cdot e = t\}$, then $e \cdot (v-v') = t-t=0$; and conversely if $e \cdot (v - v')=0$, then $e \cdot v = e \cdot v'$ and therefore $v$ and $v'$ are incident to a hyperplane orthogonal to $e$. There are $\frac{q^{d-1}-1}{q-1}$ such $e$ corresponding to every ``direction'' $e$ in a subspace isomorphic to $\mathbb{F}_q^{d-1}$. So the second summand equals $\frac{q^{d-1}-1}{q-1} (|E|-1) |E|$.

This proves the first assertion. The second assertion follows straightforwardly.
\begin{align*}
\sum_h \left(i(h) - \frac{|E|}{q}\right)^2 
& = \sum_h i(h)^2 - 2 \frac{|E|}{q} \sum_h i(h) + \frac{q(q^{d}-1)}{q-1} \frac{|E|^2}{q^2} \\
& = \sum_h i(h)^2 - 2 \frac{|E|}{q} \frac{q^{d}-1}{q-1} |E| +  \frac{q^{d}-1}{q(q-1)} |E|^2\\
& = \sum_h i(h)^2 -   \frac{q^{d}-1}{q(q-1)} |E|^2\\
& \leq \sum_h i(h)^2 -   \frac{q^{d-1}-1}{q-1} |E|^2\\
& =  q^{d-1} |E|.
\end{align*}
\end{proof}

We also prove a higher dimensional analogue of Corollary~\ref{Good direction}. 
\begin{corollary} \label{Good direction HD}
Let $d \geq 2$, $z \in \mathbb{F}_q \setminus\{0\}$, $E \subseteq \mathbb{F}_q^d$, and $\Theta \subseteq \mathbb{F}_q^{d-1}$. There exists $\theta \in \Theta$ such that
\[
|E \cdot (\theta \times \{z\})|  \geq q \, \frac{|E| |\Theta|}{q^d + |E| |\Theta|}.
\]
Hence $|E| |\Theta| > q^d$ implies there exists $\theta \in \Theta$ such that $|E \cdot  (\theta \times \{z\}) | > q/2$.
\end{corollary}

\begin{proof}

For notational convenience, given $\theta \in \mathbb{F}_q^{d-1}$ and $z \in \mathbb{F}_q$, we write $(\theta,z)$ for the vector $ \theta  \times \{z\}$.

We once again show that Lemma~\ref{var i HD} implies that $E$ is equidistributed on hyperplanes orthogonal to some direction $(\theta,z)$ for some $\theta \in \Theta$ and then apply the Cauchy-Schwartz inequality. 

Keeping in mind that $z \in \mathbb{F}_q$ is fixed, we write $v_\theta = (\theta,z)$ for each $\theta \in \Theta$ and $h_{\theta,t}$ for the hyperplane  $\{ x \in \mathbb{F}_q^d : x \cdot v_\theta = t\}$. Note that the hyperplanes $h_{\theta,t}$ are distinct. The second statement in Lemma~\ref{var i HD} implies
\[
\sum_{\theta \in \Theta} \sum_{t \in \mathbb{F}_q} \left(i(h_{\theta,t}) - \frac{|E|}{q}\right)^2 \leq \sum_{h} \left(i(h) - \frac{|E|}{q}\right)^2 \leq q^{d-1} |E|.
\] 
Therefore there exists $\theta \in \Theta$ such that 
\[
\sum_{t \in \mathbb{F}_q} \left(i(h_{\theta,t}) - \frac{|E|}{q}\right)^2 \leq \frac{q^{d-1}|E|}{|\Theta|}.
\]
Expanding the square on the left side gives
\[
\sum_{t \in\mathbb{F}_q} i(h_{\theta,t})^2 - 2 \frac{|E|}{q} \sum_{t \in \mathbb{F}_q} i(h_{\theta,t}) + \frac{|E|^2}{q} = \sum_{t \in\mathbb{F}_q} i(h_{\theta,t})^2 - 2 \frac{|E|^2}{q} + \frac{|E|^2}{q} = \sum_{t \in E \cdot v_\theta} i(h_{\theta,t})^2 - \frac{|E|^2}{q}.
\]
We used the fact that the collection of hyperplanes $\{ h_{\theta,t}\}_{t \in \mathbb{F}_q}$ partitions $\mathbb{F}_q^d$ and that $i(h_{\theta,t}) = 0$ when $t \notin E \cdot v_\theta$. Substituting above gives
\[
\sum_{t \in E \cdot v_\theta} i(h_{\theta,t})^2 \leq \frac{|E|^2}{q} + \frac{q^{d-1}|E|}{|\Theta|}.
\]
By the Cauchy-Schwarz inequality the left side is bounded below by
\[
\frac{\left( \sum_{t \in E \cdot v_\theta} i(h_{\theta,t}) \right)^2}{|E \cdot v_\theta|} = \frac{|E|^2}{|E \cdot v_\theta|}.
\]
Rearranging gives $\displaystyle |E \cdot v_\theta| \geq \frac{q |E| |\Theta|}{q^d + |E| |\Theta|}.$ 

The second conclusion follows from the fact that the above lower bound for $|E \cdot v_\theta|$ is an increasing function on the quantity $|E| |\Theta|$.
\end{proof}

\section{A good direction to project onto}
\label{Maarstrand}

The first application of Lemma~\ref{var i} is the following result of Iosevich.
\begin{theorem}[Iosevich]
Let $E \subseteq \mathbb{F}_q^2$. There exists $\xi \in \mathbb{F}_q$ such that
\[
|E \cdot (1,\xi)| \geq \frac{1}{2} \min\{|E|, q\}.
\]
\end{theorem}
\begin{proof}
Let $\Theta = \mathbb{F}_q$ in Corollary~\ref{Good direction} and deduce the existence of $\xi \in \mathbb{F}_q$ such that
\[
|E \cdot (1, \xi)| \geq \frac{q|E|}{q + |E|}.
\]
When $|E| \geq q$, the right side is at least $q/2$; and when $|E| \leq q$, the right side is at least $|E|/2$.
\end{proof}

\section{Vinh's point-line incidence theorem}
\label{Vinh}

The second application of Lemma~\ref{var i} is the following elementary proof of a theorem of Vinh. Vinh's elegant proof in \cite{Vinh2011} is based on spectral properties of regular graphs. Cilleruelo has provided an elementary proof based on Sidon sets in \cite{Cilleruelo2012}.
\begin{theorem}[Vinh]
\label{thm:2}
Let $E \subseteq \mathbb{F}_q^2$ and $\mathcal{L}$ be a collection of lines in $\mathbb{F}_q^2$. The number
\[
I(E,\mathcal{L}) = \sum_{\ell \in \mathcal{L}} i(\ell)
\] 
of point-line incidences satisfies
\[
\left| I(E, \mathcal{L}) - \frac{|E||\mathcal{L}|}{q} \right| \leq \sqrt{q |E| |\mathcal{L}|}.
\] 
Hence there is an incidence when $|E| |\mathcal{L}| > q^3$.
\end{theorem}
\begin{proof}
We combine the triangle and  Cauchy-Schwartz inequalities with the second statement in Lemma~\ref{var i}.
\begin{align*}
\left| I(E,\mathcal{L}) - \frac{|E| ||\mathcal{L}|}{q} \right| 
&= \left| \sum_{\ell \in \mathcal{L}} \left(i(\ell) - \frac{|E|}{q} \right)\right| \\
& \leq \sum_{\ell \in \mathcal{L}} \left|i(\ell) - \frac{|E|}{q} \right| \\
& \leq \sqrt{|\mathcal{L}| \sum_{\ell \in \mathcal{L}} \left(i(\ell) - \frac{|E|}{q}\right)^2 } \\
& \leq \sqrt{|\mathcal{L}| \sum_{\ell} \left(i(\ell) - \frac{|E|}{q}\right)^2 } \\
& \leq \sqrt{q |\mathcal{L}| |E|}.
\end{align*}
\end{proof}
In the $|E| |\mathcal{L}| > 2q^3$ range, Vinh's result is stronger than the point-line incidence theorem of Bourgain, Katz, and Tao in \cite{BKT2004} despite having a proof that is similar to that of the first result in the paper. In this range, Vinh's result asserts that the number of incidences is close to the case where $E$ and $\mathcal{L}$ are ``random like''.

The second conclusion of Lemma~\ref{var i HD} implies the following bound on the number of incidences $I(E,\mathcal{H})$ between a point set $E \subseteq \mathbb{F}_q^d$ and a collection of hyperplanes $\mathcal{H}$ in $\mathbb{F}_q^d$, which is also due to Vinh:
\[
\left |I(E, \mathcal{H}) - \frac{|E| |\mathcal{H}|}{q} \right| \leq \sqrt{q^{d-1} |E| |\mathcal{H}|}.
\]

\section{Pinned dot products}
\label{Pinned dot prod}

Let $E \subseteq \mathbb{F}_q^2$. Hart and Iosevich in \cite{Hart-Iosevich2008} found lower bounds on $|E|$, which guarantee that 
\[
E \cdot E =\{ u \cdot v : u,v \in E\},
\]
the set of dot products determined by $E$, is ``large''.
\begin{theorem}[Hart and Iosevich] \label{HI}
Let $E \subseteq \mathbb{F}_q^2$ and $M$ be the maximum number of points of $E$ contained in a line through the origin. 
\begin{enumerate}
\item[(i)] $E \cdot E \supseteq \mathbb{F}_q\setminus\{0\}$ provided that $|E| > q^{3/2}$.
\item[(ii)] $|E \cdot E| > q/2$ provided that $|E| > q M^{1/2}$.
\end{enumerate}
\end{theorem}
In \cite{CEHIK2012}, Chapman, Erdo{\u{g}}an, Hart, Iosevich, and Koh proved a pinned version of the above theorem by determining a condition on $E$ so that there exists $e \in E$ such that $|E \cdot e| > q/2$.  

Both sets of authors used Fourier analysis on $\mathbb{F}_q^2$ motivated by analogous results in the Euclidean setting. We explore the natural connection with point-line incidence results.

The key observation is that $\xi \in E \cdot E$ precisely when $E$ is incident to one of the lines $\ell_{e,\xi} = \{ v \in \mathbb{F}_q : v \cdot e = \xi\}$ for $e \in E$. For a fixed $\xi \neq 0$, there are $|E|$ such lines and so the first part of the above theorem follows by Vinh's point-line incidence theorem (for each $\xi \neq 0$ there is an incidence between $E$ and $\{\ell_{e,\xi}\}_{e \in E}$ when $|E|^2 > q^3$).

A pinned version of Part (ii) can be proved (in a slightly stronger form) using Lemma~\ref{var i}. Let us introduce some terminology first.

A direction $\theta \in \mathbb{F}_q$ is determined by a set $E \subseteq \mathbb{F}_q$ if the vector $(1,\theta)$ is incident to the same line through the origin as some element of $E$. In other words, $E$ determines a direction $\theta$ if $(\lambda,\lambda \theta) \in E$ for some $0 \neq \lambda \in \mathbb{F}_q$. The \emph{direction set of E}, denoted by $\mathrm{Dir}(E)$, is the set of directions determined by $E$. 
\begin{theorem}\label{Pinned dot}
Let $E \subseteq \mathbb{F}_q^2$. Suppose that $|E| |\mathrm{Dir}(E)| > q^2$.There exists $e \in E$ such that 
\[
 |E \cdot e| > \frac{q}{2}
\]
\end{theorem}
\begin{proof}
Apply Corollary~\ref{Good direction} to $E$ and $\Theta = \mathrm{Dir}(E)$. There exists $\theta \in \mathrm{Dir}(E)$ such that
\[
|E \cdot (1, \theta)| > \frac{q}{2}.
\]
As $(1, \theta) = \lambda e$ for some $e \in E$ and $0 \neq \lambda$, we get $|E \cdot e | = |E \cdot (1, \theta)| > q/2$.
\end{proof}
Note that $|\mathrm{Dir}(E)| \geq |E| /M$ and so the second part of Theorem~\ref{HI} follows.\\

A case of particular interest is $E=A \times A$ for $A \subseteq \mathbb{F}_q$. In this setting 
\[
E \cdot E = AA+AA = \{ab+cd : a,b,c,d \in A\}.
\]
Theorem~\ref{HI} implies
\begin{enumerate}
\item[(i$)^\prime$] $AA+AA \supseteq \mathbb{F}_q\setminus\{0\}$ provided that $|A| > q^{3/4}$.
\item[(ii$)^\prime$] $|AA+AA| > q/2$ provided that $|A| > q^{2/3}$ (because $M \leq |A|$).
\end{enumerate}
Both statements have not been improved since 2008. 

Statement (i$)^\prime$ seems more suitable for analytical tools. For example, using multiplicative characters in $\mathbb{F}_q\setminus\{0\}$ and the classical bounds on Jacobi sums yields the same bound $|A| > q^{3/4}$. Note that taking $q$ to be a prime congruent to 3 modulo 4 and $A$ to be the set of non-zero quadratic residues shows that 0 need not be in $AA+AA$ unless $|A|>q/2$.  

Statement (ii$)^\prime$, which is the more combinatorial of the two, has been proved in an altogether different way in fields of prime characteristic by Rudnev in \cite{Rudnev}. Using a point-plane incidence theorem in $\mathbb{F}_q^3$ similar to a line-line incidence theorem of Guth and Katz in $\mathbb{R}^2$ from \cite{Guth-Katz2015}, Rudnev established that
\[
|AA+AA| = \Omega(\min\{q, |A|^{3/2}\}) \text{ (for prime $q$)}.
\]
Rudnev's lower bound is better to that implicitly given by Hart and Iosevich for ``small sets'' that satisfy $|A| = O(q^{2/3})$. There is rich literature on the subject both for ``large sets'' and for ``small sets'' with multiplicative subgroups being a special case of particular importance (\cite{Garcia-Voloch1988,Heath-Brown-Konyagin2000,Hart-Iosevich2008,Glibichuck-Konyagin2007, Shparlinski2008,HIKR2011,Shkredov-Vyugin2012}). 

The standard in the literature is the following result of Roche-Newton, Rudnev, and Shkredov from \cite{RRS}, which also depends on ideas first developed by Guth and Katz. We state it in a simplified form that is adequate for the dot products question. 
\begin{theorem}[Roche-Newton, Rudnev, and Shkredov]\label{RRS}
Let $p$ be an odd prime and $A, \Theta \subseteq \mathbb{F}_p$. Suppose that $|A| \leq |\Theta| \leq |A|^2$. The following inequality holds.
\[
|A +\Theta A| = \Omega(\min\{p, |A| \sqrt{|\Theta|}\}).
\]
Hence by setting $\Theta = a^{-1}A$ for any $0 \neq a \in A$ gives
\[
|aA+AA| = \Omega(\min\{p, |A| \sqrt{|\Theta|}\}).
\]
\end{theorem}

Corollary~\ref{Good direction} offers a different proof of a pinned version of Theorem~\ref{RRS} for ``large sets'' in any finite field. 
\begin{theorem}
Let $A, \Theta \subseteq \mathbb{F}_q$. Suppose that $|A|^2 |\Theta| > q^2$. There exists $\theta \in \Theta$ such that
\[
|A+\theta A| > \frac{q}{2}.
\]
Hence:
\begin{enumerate}
\item[(i)] There exists $a\in A$ such that $|aA + A| > q/2$ provided that $|A|> q^{2/3}$.
\item[(ii)] There exist $a,b \in A$ such that $|aA + bA| > q/2$ provided that $|A|^2 |AA^{-1}| > q^2$.
\end{enumerate}
\end{theorem}
\begin{proof}
Apply Corollary~\ref{Good direction} to $E = A \times A$ and $\Theta$. For the special cases.
\begin{enumerate}
\item[(i)] Let $\Theta = A$.
\item[(ii)] Let $\Theta = AA^{-1}$ (or apply Theorem~\ref{Pinned dot} to $E=A\times A$ observing that $\mathrm{Dir}(E) = AA^{-1}$).
\end{enumerate}
\end{proof}
Note that for prime order fields the above theorem is much weaker than Theorem~\ref{RRS}.

A different way to bound $|AA+AA|$ is to observe that $A(A+A) \subseteq AA+AA$. Unlike $AA+AA$, which is associated with dot products in $\mathbb{F}_q^2$, the set $A(A+A)$ does not appear to be connected to lines in $\mathbb{F}_q^2$. As we describe below, passing to a ``pinned subset '' of the form $A(a+A)$ for some $a\in A$ allows us to apply Lemma~\ref{var i}.
\begin{theorem} \label{A(a+A)}
Let $A, \Theta \subseteq \mathbb{F}_q$. Suppose that $0 \notin A$ and that $|A| > q^{2/3}$. There exists $a \in A$ such that
\[
|A(a+ A)| > \frac{q}{2}.
\]
\end{theorem}
\begin{proof}
Let $\theta \in \mathbb{F}_q$. The set $A(\theta+A)$ consists of elements of the form $bc + \theta b$ for $b,c \in A$. It is therefore the projection of the set 
\[
E = \{ (bc, b) : b,c \in A\}
\]
onto the direction $(1,\theta)$. The map $(b,c) \to (bc,b)$ is a bijection between $A \times A$ and $E$ and so $|E| = |A|^2$. Applying Corollary~\ref{Good direction} to $E$ and $\Theta = A$ gives the claimed bound.
\end{proof}

Similar ideas are used in \cite{yazici2015growth} to prove the following lower bound for $|A(A+A)|$.
\begin{theorem}[Aksoy-Yazici, Murphy, Rudnev, and Shkredov]\label{YMRS}
Let $p$ be an odd prime and let $A$ be a subset of $\mathbb{F}_p$.
Then
\[
|A(A + A)| = \Omega(\min\{p, |A|^{3/2}\}).
\]
\end{theorem}
The proof of this theorem modifies the method of \cite{Rudnev,RRS} to allow coordinate transformations such as $(bc,c)\mapsto (b,c)$.

\section{Pinned algebraic distances}
\label{Iosevich-Rudnev}

Recall that the algebraic distance between two points $u=(u_1,u_2) ,v=(v_1,v_2) \in \mathbb{F}_q^2$ is defined as
\[
\| u - v\| = (u-v) \cdot (u-v) = (u_1-v_1)^2 + (u_2-v_2)^2.
\] 

There is rich literature on  the set of distances determined by a set $E \subseteq \mathbb{F}_q^2$:
\[
\Delta(E) = \{ \| u - v \| : u,v \in E \},
\]
as well as on the set of distances pinned at some $e\in E$
\[
\Delta_e(E) = \{ \| u - e \| : u \in E \}.
\]
The introduction of \cite{CEHIK2012} is an excellent reference for this type of so-called discrete Falconer questions. The state of the art can be summarised in the following theorem.

\begin{theorem}[Hanson, Lund, and Roche-Newton]\label{Perp Bis}
There exists and absolute constant $c>0$ with the following property. Let $E \subseteq \mathbb{F}_q^2$. Suppose that $|E| > q^{4/3}$. There exists $e \in E$ such that $|\Delta_e(E)| > c q$.
\end{theorem}

Our method recovers the above result for sets $E \subseteq \mathbb{F}_q^2$ that have ``many'' incidences with a line in $\mathbb{F}_q^2$. In some ways, our method is best compared with the following older theorem, which is weaker than that of Hanson, Lund, and Roche-Newton.  

\begin{theorem}[Chapman, Erdo{\u{g}}an, Hart, Iosevich, and Koh]\label{Disc Falc}
Let $E \subseteq \mathbb{F}_q^2$.
\begin{enumerate}
\item[(i)] There exists a constant $c_q$ depending only on $q$ and not on $E$ such that $|\Delta(E)| > c_q q$ provided that $|E| > q^{4/3}.$
\item[(ii)] There exists $e \in E$ such that $|\Delta_e(E)| > q/2$ provided that $|E| > q^{3/2}.$
\item[(iii)] Suppose that $E = A \times A$ for some $A \subseteq \mathbb{F}_q$. There exists $e \in A \times A$ such that $|\Delta_e(A \times A)| > q/2$ provided that $|A| > q^{2/3}.$
\end{enumerate}
\end{theorem}

We prove a result, which implies similar results to Parts (ii) and (iii) and offers some additional geometric insight on when can a large distance set be achieved. In the case of dot products, we have seen that the crucial parameter is the cardinality of the direction set of $E$. For algebraic distances an analogous role is played by the maximum number of points of $E$ incident to a single line.

\begin{theorem}\label{Pinned Dist}
Let $q$ be an odd prime power, $E \subseteq \mathbb{F}_q^2$ and $\ell$ be a line in $\mathbb{F}_q^2$. 
\begin{enumerate}
\item[(i)] Suppose that $|E| |\ell \cap E| > 2 q^2$. There exists $e \in E \cap \ell$ such that $|\Delta_e(E)| > q/2$.
\item[(ii)] Hence there exists $e \in E$ such that $|\Delta_e(E)| > q/2$ provided that $|E| > (2 q)^{3/2}$.
\item[(iii)] For the special case when $E = A \times A$ for some $A \subseteq \mathbb{F}_q$, there exists $a \in A$ such that $|\Delta_{(a,a)}| > q/2$ provided that $|A| > (2 q)^{2/3}$. 
\item[(iv)] Moreover, if $|A-A| |A|^2 > 2 q^2$, there exists $a \in A$ such that 
\[
\frac{q}{2} \leq |(a-A)^2+D^2| = |\{ (a-b)^2+s^2 : b \in A, s \in D\}| \text{ where $D=A-A$},
\]
which implies that $|\Delta(A \times A)| > q/2$.
\end{enumerate}
\end{theorem}

\begin{proof}
We follow a similar approach to the proof of Theorem~\ref{A(a+A)} in transforming $E$ in such a way that the conclusions can be deduced from Corollary~\ref{Good direction}.

(i) By applying an isometry in $\mathbb{F}_q^2$ we may assume that $\ell$ is the span of $(1,1)$. So $\ell \cap E = \{ \theta (1,1) : \theta \in \Theta\}$ for some $\Theta \subseteq \mathbb{F}_q$ of cardinality equal to $|\ell \cap E|$. In this notation we have the following.
\begin{align*}
\Delta_{(\theta, \theta)}(E) 
& = \{\|u - (\theta, \theta) \| : u \in E\} \\
& = \{\|u\| -2 u \cdot (\theta, \theta) + 2 \theta^2  : u \in E\} \\
& = \{u_1^2+u_2^2 -2 \theta (u_1 + u_2) + 2 \theta^2  : (u_1,u_2) \in E\}.
\end{align*}
Therefore
\[
|\Delta_{(\theta, \theta)}(E)|  = |\{u_1^2+u_2^2 -2 \theta (u_1 + u_2)  : (u_1,u_2) \in E\}| = |E' \cdot (1,\theta)|,
\]
where $E' = \{(u_1^2 + u_2^2, -2(u_1+u_2)) : (u_1,u_2) \in E\}$.

The next step is to prove that $|E'| \geq |E|/2$ by establishing that the map $f: E \to E'$ defined by $f((u_1,u_2)) = (u_1^2 + u_2^2, -2(u_1+u_2))$ maps at most two elements of $E$ to each element of $E'$. Indeed, if $f((u_1,u_2)) = f((w_1,w_2))$ then
\[
\left\{ \begin{matrix} u_1^2 - w_1^2 = w_2^2 - u_2^2 \\ u_1 - w_1 = w_2 - u_2 \end{matrix} \right..
\]
If $u_1 = w_1$, then the second equation gives $u_2 = w_2$. If If $u_1 \neq w_1$, then the second equation gives $u_2 \neq w_2$. Dividing the first equation by the second gives
\[
\left\{ \begin{matrix} u_1 + w_1 = w_2 + u_2 \\ u_1 - w_1 = w_2 - u_2 \end{matrix} \right.,
\]
which, together with the fact that 2 has an inverse in $\mathbb{F}_q$, implies that $u_1 = w_2$ and $u_2 = w_1$. This finishes the proof that $|E'| \geq |E| /2.$

By Corollary~\ref{Good direction}, there exists $\theta \in \Theta$ such that 
\[
|\Delta_{(\theta,\theta)}(E)| = |E' \cdot (1,\theta)| > \frac{q}{2}
\]
provided that $|E'| |\Theta| > q^2 \iff |E| |E \cap \ell| > 2 q^2$.

(ii) There are $|E| (q+1)$ incidences between $E$ and the set of all lines in $\mathbb{F}_q^2$. Moreover, there are $q(q+1)$ lines in $\mathbb{F}_q^2$. Therefore there exists a line incident to $|E| /q$ points in $E$. By Part (i) there exists $e \in E$ such that $|\Delta_e(E)| > q/2$ provided that $|E|^2 > 2 q^3.$

(iii) Follows from Part (i) by taking $\ell$ to be the span of $(1,1)$, because $|\ell \cap (A \times A)| = |A|$.

(iv) This last part cannot be deduced from Part (i) but has a very similar proof. Let $a \in A$:
\begin{align*}
|(a-A)^2+(A-A)^2| 
& = |\{ (a-b)^2+s^2 : b \in A, s \in A-A\}|\\
& = |\{a^2 - 2 ab + b^2 + s^2  : b \in A, s \in A-A\}|\\
& = |\{b^2 + s^2 - 2 ab  : b \in A, s \in A-A \}|\\
& = |E \cdot (1,a)|,    
\end{align*} 
where $E = \{(b^2 + s^2, -2b) : b \in A, s \in A-A\}$. It is straightforward to check that $|E| \geq |A| |A-A| /2$. By Corollary~\ref{Good direction} there exists $a \in A$ such that 
\[
|(a-A)^2+(A-A)^2|  = |E \cdot (1,a)| > \frac{q}{2}
\]
provided that $|E| |A| > q^2 \iff |A|^2 |A-A| > 2 q^2$.
\end{proof}

\section{Dot products and distances in higher dimensions}
\label{HD}

Questions on the number of dot products and distances determined by subsets of $\mathbb{F}_q^d$ for $d\geq2$ were investigated in detail by Chapman, Erdo{\u{g}}an, Hart, Iosevich, and Koh in \cite{CEHIK2012}. As we have seen in Section~\ref{HDim}, our method works more or less identically in $\mathbb{F}_q^d$ for all $d \geq 2$ and so can provide proofs of various results from the paper \cite{CEHIK2012}. As an illustration we prove the following theorem, which is a generalisation of a theorem of Shparlinski from \cite{Shparlinski2008}, who used multiplicative characters to establish it.

\begin{theorem}[Chapman, Erdo{\u{g}}an, Hart, Iosevich, and Koh]
Let $d \geq 2$, $A \subseteq \mathbb{F}_q$, and $z \in \mathbb{F}_q \setminus\{0\}$. Suppose that $|A| > q^{\frac{d}{2d-1}}$. There exist $a_1,\dots,a_{d-1} \in A$ such that
\[
|a_1 A + a_2 A + \dots + a_{d-1} A + z A | > \frac{q}{2}.
\] 
\end{theorem}
\begin{proof}
We apply Corollary~\ref{Good direction HD} to  $E = A \times \dots \times A$ (the $d$-fold Cartesian product of $A$ in $\mathbb{F}_q^d$) and $\Theta = A \times \dots \times A$ (the the $(d-1)$-fold Cartesian product of $A$ in $\mathbb{F}_q^{d-1}$). The hypothesis on $|A|$ implies that $|E| |\Theta| = |A|^{2d-1} > q^d$ and so there exists $v = (a_1,\dots,a_{d-1}, z)$ such that $|E \cdot v| > q/2$.
\end{proof}

\section[Generalisation of Lemma~1 to block designs]{Generalisation of Lemma~\ref{var i} to block designs}
\label{sec:block-designs}

A \emph{$(v,k,\lambda)$-block design}\/ consists of a set of points $X$ and a collection $L$ of subsets of $X$ called \emph{blocks}\/ such that
\begin{enumerate}
\item $|X|=v$,
\item $|\ell|=k$ for every block $\ell$ in $L$,
\item any two distinct points $x$ and $x'$ in $X$ are contained in exactly $\lambda$ blocks.
\end{enumerate}
We use $\ell$ to denote a typical element of $L$.
If a point $x$ in $X$ is contained in a block $\ell$ in $L$, we say that $x$ is \emph{incident}\/ to $\ell$.
See, for instance, the books \cite{jukna2011extremal, stinson2004combinatorial} for more details on combinatorial designs.

As our notation suggests, points and lines form a block design.
Specifically, if $X=\mathbb{F}_q^2$ and $L$ is the set of all lines in $\mathbb{F}_q^2$, then $(X,L)$ is a $(q^2,q,1)$-design, since $|X|=q^2$, every line contains $q$ points, and two distinct points are contained in exactly one line.

In the preceding example, not only does every line contain the same number of points, but every point is contained in the same number of lines.
In general, for any $(v,k,\lambda)$-design, there is a number $r$ (called the \emph{replication number}) such that every point in $X$ is incident to exactly $r$ blocks.

The replication number $r$ obeys two important equations:
\begin{equation}
  \label{eq:3}
  r(k-1)=\lambda(v-1).
\end{equation}
and
\begin{equation}
  \label{eq:2}
  r|X|=k|L|
\end{equation}
To prove \eqref{eq:3}, fix $x$ and let $r_x$ denote the number of blocks containing $x$; double counting pairs $(x',\ell)$ with $x'\not=x$ and $x,x'\in\ell$ shows that $r_x=\lambda(v-1)/(k-1)$ is independent of $x$. To prove \eqref{eq:2}, double count incidences.

We use $\ell(x)$ to denote the indicator function of a block $\ell$, so that $\ell(x)=1$ if $x$ is incident to $\ell$ and $\ell(x)=0$ otherwise.
If $E$ is a subset of $X$ we define the \emph{incidence function}\/ associated with $E$ by
\begin{equation}
  \label{eq:1}
  i_E(\ell)=i(\ell)=|\ell\cap E|.
\end{equation}

The following analogue of Lemma~\ref{var i} holds for any $(v,k,\lambda)$-design.
\begin{lemma}
  \label{lem:1}
Let $(X,L)$ be a $(v,k,\lambda)$-design with replication number $r$, and let $E$ be a subset of $X$.
If $i$ is the incidence function defined in \eqref{eq:1}, then
\[
\sum_{\ell\in L}i(\ell)^2 = \lambda|E|^2+(r-\lambda)|E|.
\]
Hence
\[
\sum_{\ell\in L}\left(i(\ell)-\frac{r|E|}{|L|} \right)^2\leq (r-\lambda)|E|.
\]
\end{lemma}
Note that $r|E|/|L|$ is the expected value of $i_E(\ell)$, so the second equation in Lemma~\ref{lem:1} is variance of $i_E(\ell)$, up to a factor of $1/|L|$.

The proof of Lemma~\ref{lem:1} is essentially the same as the proof of Lemma~\ref{var i}, however we need to use equations~\eqref{eq:3} and \eqref{eq:2}.

\begin{proof}[Proof of Lemma~\ref{lem:1}]
  First we compute the second moment of $i$.
\begin{align*}
\sum_\ell i(\ell)^2 
& = \sum_\ell \left( \sum_{v \in E} \ell(v) \right)^2 \\
& = \sum_\ell  \sum_{v,v' \in E} \ell(v) \ell(v') \\
& = \sum_{v \in E} \sum_\ell \ell(v) + \sum_{v\neq v' \in E} \sum_{\ell} \ell(v) \ell(v')\\
& =  r|E| + \lambda|E| (|E| - 1) \\
& = \lambda|E|^2 + (r-\lambda) |E|.
\end{align*}
In the penultimate line we used the facts that $r$ lines are incident to a point and that two distinct points are contained in $\lambda$ blocks.

Similar to the proof of Lemma~\ref{var i}, we have
\begin{align*}
\sum_\ell \left(i(\ell) - \frac{r|E|}{|L|}\right)^2 
& = \sum_\ell i(\ell)^2 - |L|\left( \frac{r|E|}{|L|}\right)^2\\
& = \sum_\ell i(\ell)^2 - \frac{r^2|E|^2}{|L|}\\
& =  r|E| + \lambda|E| (|E| - 1)  - \frac{r^2|E|^2}{|L|}\\
& = \lambda|E|^2+(r-\lambda)|E| - \frac{r^2|E|^2}{|L|}\\
& = \left(\lambda  - \frac{r^2}{|L|}\right)|E|^2+(r-\lambda)|E|.
\end{align*}
To finish the proof, we show that $\lambda - r^2/|L|\leq 0$, which implies that
\begin{equation}
  \label{eq:5}
  \sum_\ell \left(i(\ell) - \frac{r|E|}{|L|}\right)^2 = \left(\lambda  - \frac{r^2}{|L|}\right)|E|^2+(r-\lambda)|E| \leq  (r-\lambda) |E|,
\end{equation}
as desired.

By the equations for the replication number, $r/|L|=k/|X|$ and $\lambda = r(k-1)/(|X|-1)$.
Thus
\[
\lambda  - \frac{r^2}{|L|} = r\left(\frac{k-1}{|X|-1}-\frac k{|X|} \right) = r\frac{k-|X|}{|X|(|X|-1)}\leq 0,
\]
since $k\leq |X|$.
\end{proof}

\subsection*{Examples and applications}

As an application of Lemma~\ref{lem:1}, we have the following incidence theorem for $(v,k,\lambda)$-designs, first proved by Lund and Saraf \cite{lund2014incidence}.
\begin{theorem}[Lund and Saraf]
\label{thm:4}
  Let $(X,L)$ be a $(v,k,\lambda)$-design with replication number $r$.
The number of incidences between $P\subseteq X$ and $Q\subseteq L$ satisfies
\[
\left| I(P,Q)-|P||Q|r/|L|\right|\leq\sqrt{(r-\lambda)|P||Q|}
\]
\end{theorem}
The proof of Theorem~\ref{thm:4} is the same as the proof of Theorem~\ref{thm:2}, with Lemma~\ref{lem:1} in place of Lemma~\ref{var i}.

We mention some examples of block designs, and applications of Lemma~\ref{lem:1} and Theorem~\ref{thm:4}.
\begin{enumerate}
\item \emph{Points and lines}.
We have already seen that points and lines in $\mathbb{F}_q^2$ form a $(q^2,q,1)$-design.

Applying Lemma~\ref{lem:1} to this block design reproves Lemma~\ref{var i}.

\item \emph{Points and hyperplanes}.
The arguments of Section~\ref{HDim} show that the set $X$ of points in $\mathbb{F}_q^d$ and the set $\mathcal{H}$ of hyperplanes in $\mathbb{F}_q^d$ form a $(v,k,\lambda)$-design with $v=q^d$, $k=q^{d-1}$, and $\lambda = (q^{d-1}-1)/(q-1)$.
We computed directly that $r=(q^d-1)/(q-1)$, which agrees with equation \eqref{eq:3};
similarly, we computed that $|\mathcal{H}|=qr$, which agrees with equation \eqref{eq:2}.

Applying Lemma~\ref{lem:1} to this block design reproves Lemma~\ref{var i HD}.

\item \emph{Points and $m$-dimensional affine subspaces}.
In \cite{lund2014incidence}, Lund and Saraf show that the set $X$ of points in $\mathbb{F}_q^d$ and the set $L$ of $m$-dimensional affine subspaces of $\mathbb{F}_q^d$ form a $(v,k,\lambda)$-design with replication number $r$, where
\begin{align*}
  |X| &= q^d \\
  |L| &= (1+o(1))q^{m(d+1-m)}\\
  r &= (1+o(1))q^{m(d-m)}\\
k&=q^m \\
\lambda &= (1+o(1))q^{(m-1)(d-m)}.
\end{align*}
Applying Theorem~\ref{thm:4} shows that
\[
\left| I(P,Q)-|P||Q|q^{-(d-m)}\right|\leq(1+o(1))\sqrt{q^{m(d-m)}|P||Q|}
\]
for any set of points $P$ in $\mathbb{F}_q^d$ and any set $Q$ of $m$-dimensional
subspaces of $\mathbb{F}_q^d$.
\end{enumerate}

In addition, Lemma~\ref{lem:1} can be used to prove Beck's theorem and
related variations for circles \cite{IRZ2015,lund2014incidence}.

\section{Comparison to graph theoretic proofs}


We give a graph-theoretic version of the proofs above, to compare our method with the eigenvalue method used by Vinh \cite{Vinh2011} and Lund and Saraf \cite{lund2014incidence}.
The basic idea is to interpret the incidence problem graph theoretically.
Using this interpretation, we can write the quantities in Lemma~\ref{var i} in terms of the \emph{adjacency operator $A$}\/ associated to the graph.
The bounds proved in \cite{lund2014incidence,Vinh2011} use a general method for bounding the \emph{edge discrepancy}\/ of a graph in terms of the eigenvalues of the adjacency operator; this is sometimes called the ``expander mixing lemma''.

The key to the proofs in \cite{lund2014incidence,Vinh2011} and the proof we give is the explicit form of the operator $A^TA$, where $A$ is the adjacency operator associated to the block design:
\begin{equation}
  \label{eq:6}
  A^TA = (r-\lambda)I+\lambda J,
\end{equation}
where $I$ is the $|X|\times |X|$ identity matrix, and $J$ is the $|X|\times |X|$ matrix where every entry is 1.
In \cite{lund2014incidence,Vinh2011}, equation~\eqref{eq:6} is used to compute the eigenvalues of $A^TA$, which are called the \emph{singular values}\/ of $A$ (since $A$ is not necessarily a square matrix, Lund and Saraf use the singular value decomposition instead of standard eigenfunction expansion).
In addition to sketching the spectral graph theory proof, we will sketch an alternate proof using equation~\eqref{eq:6}  directly.

\subsection*{Notation and a lemma}
To begin, we establish some necessary notation.
To every incidence problem or block design, we can associate a bipartite graph $G(X,L)$ whose vertex sets are
the set of points $X$ and the set of blocks $L$.
The pair $(x,\ell)$ is an edge of $G(X,L)$ if $x\in \ell$; that is, edges correspond to incidences.

If $(X,L)$ is a $(v,k,\lambda)$-block design, then every vertex in $X$ has degree $r$ and every vertex in $L$ has degree $k$.

Let $\mathbb{F}_q^X$ and $\mathbb{F}_q^L$ denote the vector spaces of functions on $X$ and $L$, respectively.
If $x\in X$, we use $x$ to denote the indicator on $\{x\}$, so that the elements of $X$ form a basis for $\mathbb{F}_q^X$;
similarly $\ell\in L$ denote the indicator on $\{\ell\}$.
We can endow $\mathbb{F}_q^X$ with an inner product $\ip--_X$ defined by
\[
\ip fg_X:=\sum_{x\in X}f(x)g(x).
\]
We define $\ip--_L$ on $\mathbb{F}_q^L$ in the same way.

The \emph{adjacency matrix}\/ $A$ of $G$ is defined by $a_{x\ell}=1$ if $(x,\ell)$ is an edge of $G(X,L)$ and $a_{x\ell}=0$ otherwise.
Using $\ell(x)$ to denote the indicator function on the line $\ell$, we have $a_{x\ell}=\ell(x)$.

Let $E$ be a subset of $X$ and let $\chi_E$ denote the indicator function on $E$.
The incidence function $i_E(\ell)$ associated with $E$ can be expressed in terms of the adjacency matrix:
\[
i_E(\ell)=\ip{A\chi_E}{\ell\,}_L.
\]

Before we give the alternate proofs a Lemma~\ref{lem:1}, we prove a lemma that expresses the variance of $i_E(\ell)$ in terms of the adjacency matrix.
This lemma is a key step of the expander mixing lemma as well (see \cite[Theorem 9.2.5]{alon2008probabilistic}, where $i_E(\ell)$ is denoted by $N_E(\ell)$).

\begin{lemma}
\label{lem:2}
Let $e=|E|/|X|$ and let $f(x)=\chi_E(x) - e$ denote the \emph{balanced function}\/ of $E$.
Then
\[
\ip{Af}{Af}_L=\sum_{\ell\in L}\left(i(\ell)-\frac{r|E|}{|L|} \right)^2.
\]
\end{lemma}
\begin{proof}
Since
\[
ek = \frac{k}{|X|}|E|=\frac{r}{|L|}|E|
\]
by equation~\eqref{eq:2}, it is sufficient to show that
\begin{equation}
  \label{eq:adjvar}
  \ip{Af}{Af}_L=\sum_{\ell\in L}\left(i(\ell)-ek \right)^2.
\end{equation}

To prove equation~\ref{eq:adjvar}, we show that
\begin{equation}
  \label{eq:adjinc}
  \ip{Af}{\ell\,}_L=i_E(\ell)-ek,
\end{equation}
which implies the desired result by the Pythagorean identity:
\[
\ip{Af}{Af}_L=\sum_{\ell\in L}\ip{Af}{\ell\,}_L^2 =\sum_{\ell\in L} (i_E(\ell)-ek)^2.
\]
Finally, equation~\ref{eq:adjinc} follows from direct computation:
\begin{align*}
  \ip{Af}{\ell\,}_L&=  \ip{A(\chi_E-e\chi_X)}{\ell\,}_L\\
&= \ip{A\chi_E}{\ell\,}_L -  e\ip{A\chi_X}{\ell\,}_L\\
&=i_E(\ell)-ek.
\end{align*}
\end{proof}

\subsection*{Alternate proof by spectral graph theory}
From here, we can finish the spectral graph theory proof of \cite{Vinh2011} and \cite{lund2014incidence} in roughly three steps.
First, using the singular value decomposition, we could show that
\begin{equation}
  \label{eq:7}
  \sum_{\ell\in L}\left(i_E(\ell)-ek\right)^2\leq \lambda_2 e(1-e)|X|=\lambda_2(1-e)|E|,
\end{equation}
where $\lambda_2$ is the singular value of $A$ with the second largest magnitude.
This is equivalent to the equation
\[
\Vert Af\Vert_2^2=\ip{Af}{Af}_L \leq \lambda_2\ip{f}{f}_X,
\]
which says that the $L^2$ operator norm of $A$ restricted to the space of balanced functions is bounded by $\lambda_2$, combined with the equation
\[
\lambda_2\ip{f}{f}_X= \lambda_2 e(1-e)|X| = \lambda_2 (1-e)|E|.
\]

Second, we compute $\lambda_2$ for a $(v,k,\lambda)$-design with replication number $r$---by the explicit form for $A^TA$ (equation \eqref{eq:6}) we have $\lambda_2=r-\lambda$.

Finally, we combine $\lambda_2=r-\lambda$ with Lemma~\ref{lem:2} and equation~\eqref{eq:7} to prove Lemma~\ref{var i}:
\[
\sum_{\ell\in L}(i_E(\ell)-ek)^2\leq (r-\lambda) e(1-e)|X|\leq (r-\lambda)e|X| = (r-\lambda)|E|.
\]

\subsection*{Alternate proof by direct computation}
It is possible deduce Lemma~\ref{var i} directly from Lemma~\ref{lem:2} and equation~\eqref{eq:6}, without using the singular value decomposition.
Here is the key claim.
\begin{claim}
If $A$ is the adjacency matrix associated to the $(v,k,\lambda)$-design $(X,L)$ and $f(x)=\chi_E-e$ is the balanced function of a set of points $E\subseteq X$, then
\[
A^TAf(x) = (r-\lambda)f(x),
\]
where $r$ is the replication number of $(X,L)$.
\end{claim}
\begin{proof}
  By equation~\eqref{eq:6}, we have
\[
A^TAf(x)=(r-\lambda)If(x)+\lambda Jf(x).
\]
The claim follows by showing that $f$ is annihilated by $J$: for all $x$,
\[
Jf(x) = \sum_{x'\in X}f(x') = \sum_{x'\in X}(\chi_E(x')-e) = |E|-e|X| = 0,
\]
since $e=|E|/|X|$.
\end{proof}

Now we may compute $\ip{Af}{Af}_L$ \emph{exactly}:
\[
\ip{Af}{Af}_L = \ip{A^TAf}{f}_X  = \ip{(r-\lambda)f}{f}_L = (r-\lambda)\ip{f}{f}_X.
\]
All together,
\begin{align*}
  \sum_{\ell\in L}\left(i(\ell)-\frac{r|E|}{|L|} \right)^2=\ip{Af}{Af}_L
 &= (r-\lambda)\ip{f}{f}_X \\
&= (r-\lambda)(1-e)|E|\leq (r-\lambda)|E|,
\end{align*}
which reproves Lemma~\ref{lem:1}.

Finally, to complete the connection to our original argument, we derive equation~\eqref{eq:6}, which states that
\[
A^TA=(r-\lambda)I+\lambda J.
\]
That is, the diagonal entries of $A^TA$ are $r$ and the off-diagonal entries are $\lambda$.
Since $\ip{Ax}{Ax'}_L$ counts the number of lines incident to both $x$ and $x'$, we have
\[
\ip{A^TAx}{x'}_X=\ip{Ax}{Ax'}_L=
\begin{cases}
  r & \mbox{ if } x=x'\\
\lambda &\mbox{ if } x\not=x',\\
\end{cases}
\]
as desired.

\section{Comparison to Cilleruelo's Sidon set argument}

In this final section, we compare our method to another elementary method based on \emph{Sidon sets}, introduced by Cilleruelo \cite{Cilleruelo2012}.
Cilleruelo used Sidon sets to give an alternate proof of Vinh's incidence bound (Theorem~\ref{thm:2}).
We will see that our method is closely related to Cilleruelo's method.

A subset $S$ of an abelian group $G$ is called a \emph{Sidon set}\/ if any non-zero element of $G$ can be written as a difference of elements in $S$ in at most one way.
We write
\[
r_{A-B}(x) = |A \cap (x+B)|
\]
for the number of ways to express $x$ as a difference $a-b$ with $a$ in $A$ and $b$ in $B$.
In this notation, a Sidon set satisfies
\[
r_{S-S}(x)\leq 1
\]
for all $x\not=0$ in $G$.

Since
\[
|S|^2 = \sum_{x\in G}r_{S-S}(x)\leq |S|+|G|-1,
\]
we have $|S|\leq\sqrt{|G|}+1/2$.
The most interesting Sidon sets are those with $|S|=\sqrt{|G|}-\delta$ for some small quantity $\delta$.
In this setting Cilleruelo \cite{Cilleruelo2012} proved the following theorem.
\begin{theorem}[Cilleruelo]
\label{thm:1}
  Let $\delta$ be a constant and $S$ be a Sidon set in a finite abelian group $G$ with $|S|=\sqrt{|G|}-\delta$.
Then for all $A,B\subseteq G$, we have
\[
|\{(a,b)\in A\times B\colon a+b\in S\}|=\frac{|S|}{|G|}|A||B|+\theta(|A||B|)^{1/2}|G|^{1/4}
\]
with $|\theta|< 1+ \max\{\delta,0\} \frac{ |A|}{|G|}$.
\end{theorem}

The proof of Theorem~\ref{thm:1} hinges on a bound that is extremely similar to the bounds in Lemmas~\ref{var i}, \ref{var i HD}, and \ref{lem:1}:
\begin{equation}
  \label{eq:4}
  \sum_{x\in G}\left(r_{S-A}(x)-\frac{|S||A|}{|G|}\right)^2\leq |A|(|S|-1)+|A|^2 \left(1-\frac{|S|^2}{|G|}\right),
\end{equation}
c.f. Equation~(2.3) in \cite{Cilleruelo2012}. This coincidence can be explained: we can associate a block diagram to $G$ and $S$, which is not quite a $(v,k,\lambda)$-design, but is close to one.

Let the set of points $X$ be $G$ and let the set of blocks $L$ be the subsets of $G$ of the form $y-S$, with $y$ in $G$.
Thus $|X|=|L|=|G|$ and each block has size $k=|S|$.
Since $x\in y-S$ if $y\in x+S$, we see that the replication number $r$ is equal to $|S|$ as well.
Note that the incidence graph associated to this block design is simply the Cayley graph of $G$ defined by the Sidon set $S$.

This is not a $(v,k,\lambda)$-design, since each pair of points $x,x'$ in $X$ is contained in \emph{at most}\/ one block.
To prove this, note that
\begin{align*}
  |\{y-S \colon x,x'\in y-S\}|&=|\{y\in G\colon x+s=x'+s'=y,\mbox{ for some $s,s'$ in $S$}\}|\\
& = r_{S-S}(x-x')\leq 1.
\end{align*}
This says that if we let $\lambda_{x,x'}$ denote the number of blocks that contain both $x$ and $x'$, then we have $0\leq\lambda_{x,x'}\leq 1$, but $\lambda_{x,x'}$ is not independent of the pair $x,x'$ (in fact, $\lambda_{x,x'}=r_{S-S}(x-x')$).

Now we make the connection between equation~\eqref{eq:4} and our previous arguments.
If $A$ is a subset of $G$, then the incidence function $i_A(y)$ associated with $A$ is given by
\[
i_A(y)=r_{S-A}(y),
\]
thus
\[
\sum_{x\in G}\left(r_{S-A}(x)-\frac{|S||A|}{|G|}\right)^2 = \sum_{x\in G}\left(i_A(x)-\frac{r|A|}{|L|}\right)^2.
\]
Just as in the proof of Lemma~\ref{lem:1}, we have
\begin{align*}
 \sum_{x\in G}\left(i_A(x)-\frac{r|A|}{|L|}\right)^2&\leq r|A| +\lambda|A|(|A|-1)-\frac{r^2|A|^2}{|L|}\\
&= |A||S|+|A|(|A|-1)-\frac{|A|^2|S|^2}{|G|}\\
& = |A|(|S|-1)+|A|^2 \left( 1-\frac{|S|^2}{|G|}\right),
\end{align*}
where we have set $\lambda=1$ and used $\lambda_{x,x'}\leq\lambda$.

Section~2 in \cite{Cilleruelo2012} details how Theorem~\ref{thm:1} follows from \eqref{eq:4}. For the benefit of the reader we mention that, in contrast to previous arguments, one cannot get rid of the $|A|^2$ term because its coefficient might be positive. In the language we have developed this is explained as follows. We only have an upper bound on $\lambda_{x,x'}$. This implies the inequality $\lambda(v-1)\geq r(k-1)$, which goes in the wrong direction and cannot be used as a substitute to equation~\eqref{eq:3}. Instead, one must use the fact $|S|=\sqrt{|G|}-\delta$ for some $\delta \geq -1/2$ and do some algebra.
\begin{unremark}
  Suppose that in general $\lambda_{x,x'}\leq\lambda$.
The same argument shows that if $|S|=\sqrt{\lambda |G|}$ plus lower order terms, then we can still cancel the terms involving $|A|^2$.

For instance, we could prove Lemma~\ref{var i} by noting that the incidence graph of points and lines is $K_{2,2}$ free, hence $\lambda_{x,x'}\leq1$, and that $r=q=\sqrt{|X|}$.
\end{unremark}

\phantomsection

\addcontentsline{toc}{section}{References}

\bibliography{all}

\hspace{20pt} Department of Mathematics, University of Rochester, New York, USA.

\hspace{20pt} \textit{Email addresses}: \href{mailto:bmurphy8@ur.rochester.edu}{bmurphy8@ur.rochester.edu} and \href{mailto:giorgis@cantab.net}{giorgis@cantab.net}

\end{document}